\documentclass[11pt,a4paper]{article}
\usepackage[latin1]{inputenc}
\usepackage[T1]{fontenc}
\usepackage[american]{babel}      % language
\usepackage{amsmath,amssymb,amsthm}
\usepackage{graphicx}
\usepackage{authblk}
\newcommand{\mau}{\geq}
\newcommand{\miu}{\leq}

\newcommand{\N}{\mathbb{N}}
\newcommand{\R}{\mathbb{R}}

\newcommand{\Z}{\mathbb{Z}}

\newcommand{\uu}{\underline{u}}

\newcommand{\disp}{\displaystyle}

\newcommand{\xx}{\underline{x}}

\newcommand{\kk}{\underline{k}}
\newcommand{\tk}{t_{\underline{k}}}
\newcommand{\Rkw}{R^w_{\underline{k}}}
\newcommand{\Ak}{A_{\underline{k}}}

\newcommand{\eps}{\epsilon}
\newcommand{\ttt}{\underline{t}}

\newtheorem{definition}{Definition}[section]
\newtheorem{remark}[definition]{Remark}
\newtheorem{theorem}[definition]{Theorem}
\newtheorem{lemma}[definition]{Lemma}

\newtheorem{corollary}[definition]{Corollary}

\makeatletter
    \let\@fnsymbol\@arabic
    \makeatother
\title{Rate of approximation for multivariate sampling Kantorovich operators on some functions spaces}
\author[1]{{\bf Danilo Costarelli}}
\author[2]{{\bf Gianluca Vinti}}
\affil[1]{\small{Dipartimento di Matematica e Fisica, Sezione di Matematica, Universit\`{a} degli Studi  Roma Tre, 1, Largo S. Leonardo Murialdo, 00146 Rome, Italy, {\tt danilo.costarelli@gmail.com}}}
\affil[2]{\small{Dipartimento di Matematica e Informatica, Universit\`{a} degli Studi di Perugia,
\vskip0.001cm 1, Via Vanvitelli, 06123, Perugia, Italy, {\tt gianluca.vinti@unipg.it}}}
 
\date{}

\begin{document} 
\maketitle

\abstract{%
In this paper, the problem of the order of approximation for the multivariate
sampling Kantorovich operators is studied. The cases of
the uniform approximation for uniformly continuous and bounded functions/signals belonging to Lipschitz classes and the case of the modular approximation for
functions in Orlicz spaces are considered. In the latter context, Lipschitz classes of Zygmund-type which take into account of the modular functional involved are introduced. Applications to $L^p(\R^n)$, interpolation and exponential spaces can be deduced from the general theory formulated in the setting of Orlicz spaces. The special cases of multivariate sampling Kantorovich operators based on kernels of the product type and constructed by means of Fej\'er's and B-spline kernels have been studied in details.
}

\section{Introduction}

The sampling Kantorovich operators have been introduced to approximate and reconstruct not necessarily continuous signals.
The multivariate sampling Kantorovich operators considered in this paper (\cite{COVI1}) are of the form:
$$
(S_w f)(\underline{x})\ :=\ \sum_{\underline{k} \in \Z^n} \chi(w\underline{x}-t_{\underline{k}})\left[\frac{w^n}{A_{\underline{k}}} \int_{R_{\underline{k}}^w}f(\underline{u})\ d\underline{u}\right], \hskip0.7cm (\underline{x} \in \R^n), \hskip0.8cm \mbox{(I)}
$$
where $f: \R^n \to \R$ is a locally integrable function such that the above series is convergent for every $\underline{x} \in \R^n$. The symbol $t_{\underline{k}}=\left(t_{k_1},...,t_{k_n}\right)$ denotes vectors where each $(t_{k_i})_{k_i \in \Z}$, $i=1,...,n$ is a certain strictly increasing sequence of real numbers with $\Delta_{k_i}=t_{k_{i+1}}-t_{k_i}>0$.
Note that the sequences $(t_{k_i})_{k_i \in \Z}$ are not necessary equally spaced (irregular sampling scheme). We denote by $R_{\underline{k}}^w$ the sets:
$$
R_{\underline{k}}^w\ :=\ \left[\frac{t_{k_1}}{w},\frac{t_{k_1+1}}{w}\right]\times\left[\frac{t_{k_2}}{w},\frac{t_{k_2+1}}{w}\right]\times...\times\left[\frac{t_{k_n}}{w},\frac{t_{k_n+1}}{w}\right], \hskip2.2cm \mbox{(II)}
$$
$w>0$ and $A_{\underline{k}} = \Delta_{k_1} \cdot \Delta_{k_2} \cdot...\cdot \Delta_{k_n}$, $\underline{k} \in \Z^n$. Moreover, the function $\chi:\R^n \to \R$ is a kernel satisfying suitable assumptions. 
In \cite{BABUSTVI}, the authors introduced these operators in the univariate form, starting from the well-known generalized sampling operators (see e.g. \cite{RIST,BURIST2,BUFIST,BUST2,CV1,BAVI2,VI0,VI1,BAVI7,BABUSTVI0,BABUSTVI2}) and replacing, in their definition, the sample values $f(k/w)$ with $w \int_{k/w}^{(k+1)/w}f(u)\,du$. Clearly, this is the most natural mathematical modification to obtain operators which can be well-defined also for general measurable, locally integrable functions, not necessarily continuous. Moreover, this situation very often occurs in Signal Processing, when one cannot match exactly the sample at the point $k/w$: this represents the so-called "time-jitter'' error. The theory of sampling Kantorovich operators allow us to reduces the time-jitter error, calculating the information in a neighborhood of $k/w$ rather that exactly in the node $k/w$. These operators, as the generalized sampling operators, represent an approximate version of classical sampling series, based on the Whittaker-Kotelnikov-Shannon sampling theorem (see e.g. \cite{JER,DODSI,BUST1,BEKA,HIG2,HIST,ANVI}).

  Subsequently, the sampling Kantorovich operators have been studied in various settings. In \cite{VIZA1,COVI2} the nonlinear univariate and multivariate versions of these operators were introduced; applications to image processing have been discussed in \cite{COVI1,COVI2}. Indeed, static gray scale images are characterized by jumps of gray levels mainly concentrated in their contours or edges and this can be translated, from a mathematical point of view, by discontinuities. For these reasons, multivariate sampling Kantorovich operators appear very appropriate for applications to image reconstruction.  Moreover, some preliminary applications to civil engineering have been presented in \cite{CLCOMIVI1,CLCOMIVI3}. Results concerning the order of approximation have been obtained in \cite{COVI3,COVI4} in the univariate case, for the linear and nonlinear versions of these operators. Extensions of the theory to more general contexts were presented in \cite{DOVI,VIZA2,BAMA2,VEVI1,VIZA3}.

   In this paper, we study the problem of the rate of approximation for the multivariate
sampling Kantorovich operators in various settings. More precisely, we consider the case of
the uniform approximation for uniformly continuous and bounded functions belonging to Lipschitz classes and the case of the modular approximation for
functions in Orlicz spaces $L^{\varphi}(\R^n)$. In this context, we will introduce Lipschitz classes of
Zygmund-type which take into account of the modular functional involved. From the results concerning Orlicz spaces, applications to $L^p(\R^n)$ spaces, $1 \miu p < +\infty$, as to other examples of Orlicz spaces can be deduced. In particular, the application of the multivariate sampling Kantorovich operators to $L^p$-spaces is suitable for Signal/Image Processing. Other important cases of Orlicz spaces here considered are the interpolation spaces $L^{\alpha}\log^{\beta}L(\R^n)$ and the exponential spaces, which are very used for applications, e.g., to partial differential equations and for embedding theorems between Sobolev spaces respectively. 

   In order to obtain results concerning the order of approximation for the multivariate sampling Kantorovich operators starting from the one-dimensional theory, some difficulties arise. First of all, the definition of the Lipschitz class in which we work must be extended to the multivariate case, both in case of uniformly continuous functions and for functions in Orlicz spaces. But the main difficulty is related to the construction of multivariate kernels which satisfy all the assumptions of the above theory (see Section \ref{applications}). Indeed, it turns out that the kernels affect the rate of approximations when uniform and modular approximation are considered. For this reason we introduce a procedure useful to construct multivariate kernel and to determine their parameters $\mu$, $\beta$ and $\alpha$, starting from one-dimensional kernels.
The special cases of multivariate kernels of the product type, based upon Fej\'er's and B-spline kernels have been studied in details.

%%%%%%%%%%%%%%%%%%%%%%%%%

\section{Preliminary notions} \label{sec2}

In this paper, we will denote by $C(\R^n)$ the set of all uniformly continuous and bounded functions $f: \R^n \rightarrow \R$ endowed with the usual sup-norm $\|\cdot\|_{\infty}$. Moreover, by $\| \cdot \|_2$ we will denote the usual Euclidean norm in $\R^n$, i.e., $\|\xx\|_2 := (x_1^2 + ... + x_n^2)^{1/2}$, $\xx = (x_1,...,x_n) \in \R^n$. 

  In order to study the rate of approximation of a family of linear operators, we introduce the definition of the Zygmund-type class (Lipschitz class) for functions of several variables.

  We define the class $Lip_{\infty}(\nu)$, $0<\nu \miu 1$, as
$$
Lip_{\infty}(\nu)\ :=\ \left\{f \in C(\R^n):\ \left\|f(\cdot)-f(\cdot + \ttt)\right\|_{\infty} =\mathcal{O}(\|\ttt\|_2^{\nu}),\ \mbox{as}\ \|\ttt\|_2 \rightarrow0 \right\},
$$
where for any two functions $f$, $g: \R^n\rightarrow\R$, $f(\ttt)=\mathcal{O}(g(\ttt))$ as $\|\ttt\|_2\rightarrow0$ means that there exist constants $C$, $\gamma>0$ such that $\left|f(\ttt)\right| \miu C\left|g(\ttt)\right|$ for every $\ttt \in \R^n$, with $\|\ttt\|_2 \miu \gamma$ (\cite{BAMUVI,VI1}). The above definition represents the extension to the multivariate setting of the Zygmund-type classes introduced in \cite{COVI3} for univariate functions. It is easy to observe that, in case of functions defined on bounded domains, the above definition is equivalent to the well-known definition of $\nu$-Holder continuous functions.

The sampling Kantorovich operators $S_w$ studied in this paper are suitable to reconstruct not necessarily continuous signals (see, e.g., \cite{BABUSTVI,COVI1}), and a very general class of spaces containing such functions are the so-called Orlicz spaces. Since in the rest of the paper we will study the above operators also in this general setting, we now
recall some basic notions concerning Orlicz spaces.

  A function $\varphi: \R^+_0 \rightarrow \R^+_0$ is said to be a $\varphi$-function if it satisfies the following conditions:\\
$(\Phi 1)$ $\varphi$ is a non decreasing and continuous function;\\
$(\Phi 2)$ $\varphi(0)=0$, $\varphi(u)>0$ if $u >0$ and $\lim_{u \rightarrow +\infty} \varphi(u)=+\infty$.\\
Let us now consider the functional $I^{\varphi}$ associated to any given $\varphi$-function $\varphi$ and defined as follows
$$
I^{\varphi}[f]\ :=\ \int_{\R^n}\varphi(|f(\xx)|)\ d\xx,
$$
for every $f \in M(\R^n)$, i.e., for every (Lebesgue) measurable function $f: \R^n \rightarrow \R$. As it is well-known, the functional $I^{\varphi}$ satisfy a number of important properties. For instance, $I^{\varphi}$ is a modular functional (see e.g. \cite{BAMUVI,MU1,RAO1}), and moreover, if the $\varphi$-function $\varphi$ is convex, the corresponding modular functional is convex too.

   Now, we are able to recall the definition of the Orlicz space $L^{\varphi}(\R^n)$ generated by $\varphi$. We define
$$
L^{\varphi}(\R^n)\ :=\ \left\{ f \in M(\R^n)\ :\ I^{\varphi}[\lambda f] < +\infty,\ \mbox{for some}\ \lambda>0\right\}.
$$
A notion of convergence in Orlicz spaces, called {\em modular convergence}, was introduced in \cite{MUORL}, which induces a topology in $L^{\varphi}(\R^n)$, called {\em modular topology}.

  A family $(f_w)_{w>0} \subset L^{\varphi}(\R^n)$ is said to be modularly convergent to $f \in L^{\varphi}(\R^n)$, if there exists $\lambda>0$ such that
\begin{equation} \label{def_mod_convergence}
I^{\varphi}[\lambda(f_w-f)]\ =\ \int_{\R^n}\varphi(\lambda|f_w(\xx)-f(\xx)|)\ d\xx\ \longrightarrow 0, \hskip0.5cm as \hskip0.5cm w \rightarrow +\infty. 
\end{equation}
Moreover we recall, for the sake of completeness, that in $L^{\varphi}(\R^n)$ can be also given a stronger notion of convergence, i.e., the Luxemburg-norm convergence, see e.g. \cite{BAMUVI,MU1}. 

  We now define by $Lip_{\varphi}(\nu)$, $0<\nu \miu 1$, the Zygmung-type class in Orlicz spaces, as the set of all functions $f \in M(\R^n)$ such that there exists $\lambda>0$ with 
$$
I^{\varphi}[\lambda(f(\cdot)-f(\cdot+\ttt))]=\int_{\R^n}\varphi\left(\lambda\left|f\left(\xx\right)-f\left(\xx + \ttt\right)\right|\right) d\xx\ =\ \mathcal{O}(\|\ttt\|_2^{\nu}),
$$
as $\|\ttt\|_2 \to 0$.  The above definition extends that one given in \cite{COVI3} for functions of one variable. For further results concerning Orlicz spaces see \cite{BAMA,BAMUVI,BAVI_1,KOZ,KRA,MAL,MU1,RAO1,RAO2}.

%%%%%%%%%%%%%%%%%%%

\section{Multivariate sampling Kantorovich operators} \label{sec3}

In this section, the definition of the multivariate sampling Kantorovich operators is recalled (\cite{COVI1}), together with its main properties and some useful remarks.

  Let $\Pi^n =(\tk)_{\kk \in \Z^n}$ be a sequence of real numbers defined by $\tk = (t_{k_1}, ..., t_{k_n})$, where each $(t_{k_i})_{k_i \in \Z}$, $i=1,...,n$, is a sequence such that $-\infty < t_{k_i} < t_{k_i+1} < +\infty$ for every $k_i \in \Z$, $\lim_{k_i\rightarrow \pm \infty}\ t_{k_i}\ =\ \pm \infty$ and there are two positive constants $\Delta$, $\delta$ such that $\delta \miu \Delta_{k_i}:= t_{k_i+1}-t_{k_i} \miu \Delta$, for every $k_i \in \Z$.\\
In what follows, a function $\chi: \R^n\rightarrow\R$ will be called a kernel if it satisfies the following properties:
\begin{itemize}
	\item ($\chi 1$) $\chi \in L^1(\R^n)$ and is bounded in a neighborhood of the origin;
	\item ($\chi 2$) For some $\mu >0$,
$$
\sum_{\kk \in \Z^n}\chi(w\xx-\tk)-1\ =:\ A_w(\xx)-1 =\  \mathcal{O}(w^{-\mu}),\hskip0.5cm as\ \hskip0.5cm w \to +\infty,
$$
uniformly with respect to $\xx \in \R^n$;
	\item ($\chi 3$) For some $\beta>0$, we assume that the discrete absolute moment of order $\beta$ is finite, i.e.,
	$$
	m_{\beta,\Pi^n}(\chi)\ :=\ \sup_{\uu \in \R} \sum_{\kk \in \Z^n}\left|\chi(\uu-\tk)\right|\cdot\|\uu-\tk\|^{\beta}\ <\ +\infty;
$$
   \item ($\chi 4$) There exists $\alpha>0$ such that, for every $M>0$,
$$ 
\int_{\| \uu \|_2 > M} w^n \left| \chi(w\, \uu)\right|\ d\uu\ =\ \mathcal{O}(w^{-\alpha}), \hskip0.5cm as \hskip0.5cm w \rightarrow +\infty.
$$
\end{itemize}
The conditions listed above are the typical properties satisfied by the approximate identities, and are standard assumptions required in case of discrete linear operators.

 In order to recall the definition of sampling Kantorovich operators, we first introduce the following notation:
$$
R^w_{\kk}\, :=\, \left[ \frac{t_{k_1}}{w}, \frac{t_{k_1+1}}{w} \right] \times ... \times \left[ \frac{t_{k_n}}{w}, \frac{t_{k_n+1}}{w} \right]\ \subset \R^n,
$$
for every $\kk \in \Z^n$ and $w>0$. Denoting by $A_{\kk}:= \Delta_{k_1} \cdot ... \cdot \Delta_{k_n}$, the Lebesgue measure of $R^w_{\kk}$ is given by $A_{\kk}/w^n$. We now define by $(S_w)_{w>0}$ the family of the multivariate sampling Kantorovich operators defined by
\begin{equation} \label{kantorovich}
(S_w f)(\xx)\ :=\ \sum_{\kk \in \Z^n}\chi(w\xx-\tk)\left[\frac{w^n}{A_{\kk}}\, \int_{R^w_{\kk}}f(\uu)\ d\uu\right] \hskip1cm (\xx \in \R^n),
\end{equation}
where $f:\R^n \rightarrow\R$ is a locally integrable function such that the series is convergent for every $\xx \in \R^n$.

  We begin giving the proof of following lemma.
\begin{lemma} \label{lemma1}
Under the assumptions $(\chi 1)$ and $(\chi  3)$ on the kernel $\chi$, we have\\
(i) $\disp m_{0,\Pi^n}(\chi)\ :=\ \sup_{\uu \in \R^n} \sum_{\kk \in \Z^n}\left|\chi(\uu-\tk)\right|\ <\ +\infty$; \\
(ii) For every $\gamma>0$,
$$
\sum_{\|w\xx-\tk \|_2 >\gamma w}\left|\chi(w\xx-\tk)\right|\ =\ \mathcal{O}(w^{-\beta}), \hskip0.5cm as \hskip0.5cm w\rightarrow +\infty,
$$
uniformly with respect to $\xx \in \R^n$, where $\beta>0$ is the constant of condition $(\chi 3)$.
\end{lemma}
\begin{proof}
For a proof of (i) see e.g. \cite{COVI1}.\\
(ii) Let $\gamma >0$ be fixed. For every $\xx \in \R^n$ and $w>0$ we obtain
\begin{eqnarray*}
\sum_{\|w\xx-\tk \|_2 >\gamma w}\left|\chi(w\xx-\tk)\right|\ &\miu&\ \frac{1}{\gamma^{\beta}w^{\beta}}\sum_{\|w\xx-\tk \|_2 >\gamma w}\left|\chi(w\xx-\tk)\right|\cdot \|w\xx-\tk\|_2^{\beta}\ \miu\ \\
\\
&\miu&\ \frac{1}{\gamma^{\beta}w^{\beta}}\, m_{\beta,\Pi^n}(\chi)\ <\ +\infty,
\end{eqnarray*}
and so the assertion follows.\\
\end{proof}
\begin{remark} \label{osservazione1} \rm
In case of $f \in L^{\infty}(\R^n)$, by Lemma \ref{lemma1} (i), $S_w f$ are well-defined for every $w>0$. Indeed,
\begin{displaymath}
\left|(S_w f)(\xx)\right|\ \miu\ m_{0, \Pi^n}(\chi)\left\|f\right\|_{\infty}\ < +\infty,
\end{displaymath}
for every $\xx \in \R^n$ and $w>0$, i.e., $S_w : L^{\infty}(\R^n)\rightarrow L^{\infty}(\R^n)$.
\end{remark}

\section{Order of approximation in $C(\R^n)$} \label{sec4}

We now begin by studying the rate of approximation for the family of linear, multivariate sampling Kantorovich operators (\ref{kantorovich}) in $C(\R^n)$. 
\begin{theorem} \label{theorem1}
Let $\chi$ be a kernel and $f \in Lip_{\infty}(\nu)$, $0<\nu \miu 1$. Then
$$
\left\|S_w f-f\right\|_{\infty}\ =\ \mathcal{O}(w^{-\eps}), \hskip0.5cm as \hskip0.5cm w\rightarrow +\infty,
$$
where $\eps := \min\left\{\nu,\ \beta,\ \mu \right\}$ and $\mu$, $\beta>0$ are the constants of conditions $(\chi 2)$ and $(\chi 3)$, respectively.
\end{theorem}
\begin{proof}
First, we consider the case of $\chi$ satisfying condition $(\chi 3)$ for $0 < \beta \miu 1$.
\vskip0.1cm

  Let now $f \in Lip_{\infty}(\nu)$, $0<\nu \miu \beta$, be fixed. By Remark \ref{osservazione1}, $S_w f$ are well-defined for every $w>0$. Moreover, since $f \in Lip_{\infty}(\nu)$, we have
$$
\sup_{\xx \in \R^n}\left|f(\xx)-f(\xx+\ttt)\right|\ \miu\ C  \|\ttt\|_2^{\nu},
$$
for some constants $C$, $\gamma>0$, and
for every $\ttt \in \R^n$ such that $\| \ttt\|_2 \miu \gamma$. Let now $\xx \in \R^n$ be fixed. Then we can write
\begin{eqnarray*}
&& \left|(S_w f)(\xx)-f(\xx)\right|\ \miu \left|(S_w f)(\xx)-f(\xx)A_w(\xx)\right|\ +\ \left|f(\xx)A_w(\xx)-f(\xx)\right|\\
&& \miu\ \sum_{\kk \in \Z^n}\left|\chi(w\xx-\tk)\right|\frac{w^n} {A_{\kk}}\int_{R^w_{\kk}}\left|f(\uu)-f(\xx)\right|\ d\uu\ +\ |f(\xx)|\, |A_w(\xx)-1| \\
&\miu&\ \left(\sum_{\|w\xx-\tk\|_2 \miu w\gamma/2}+\sum_{\|w\xx-\tk\|_2 > w\gamma/2}\right)\left|\chi(w\xx-\tk)\right|\frac{w^n} {A_{\kk}}\int_{R^w_{\kk}} \left|f(\uu)-f(\xx)\right|\ d\uu \\
&+& \|f\|_{\infty} |A_w(\xx)-1| =:\ I_1\ +\ I_2\ +\ I_3.
\end{eqnarray*}
In order to estimate $I_1$ we now introduce the following notation. We denote by $(R^w_{\kk}-\xx)$ the sets of the form
$$
(R^w_{\kk}-\xx)\ :=\ \left[\frac{t_{k_1}}{w}-x_1,\, \frac{t_{k_1+1}}{w}-x_1\right] \times ...\times \left[\frac{t_{k_n}}{w}-x_n,\, \frac{t_{k_n+1}}{w}-x_n\right].
$$
We can observe that 
for every $\ttt \in (R^w_{\kk}-\xx)$, 
if $\|w\xx-\tk\|_2 \miu w \gamma/2$, we have
$$
\|\ttt\|_2\ \miu\ \|\ttt-\tk/w+\xx\|_2\, +\, \|\tk/w-\xx\|_2\ \miu\ \sqrt n\ \frac{\Delta}{w}\, +\, \frac{\gamma}{2}\ <\ \gamma,
$$
for sufficiently large $w>0$. Then by the change of variable $\uu=\xx+\ttt$ in the integrals of $I_1$, the above inequality and the definition of $Lip_{\infty}(\nu)$, we can obtain
\begin{eqnarray}
I_1\ &=&\ \sum_{\|w\xx-\tk\|_2 \miu w\gamma/2}\left|\chi(w\xx-\tk)\right|\, \frac{w^n}{A_{\kk}}\int_{(R^w_{\kk}-\xx)}\left|f(\xx+\ttt)-f(\xx)\right|\ d\ttt  \nonumber \\
\label{stima_prop} \\
&\miu&\ C \!\! \sum_{\|w\xx-\tk\|_2 \miu w\gamma/2}\left|\chi(w\xx-\tk)\right|\ \frac{w^n}{A_{\kk}}\int_{(R^w_{\kk}-\xx)} \|\ttt\|_2^{\nu}\ d\ttt \nonumber.   
\end{eqnarray}
In order to estimate (\ref{stima_prop}), we proceed as follows:
$$
\max_{\ttt \in (R^w_{\kk}-\xx)}\|\ttt\|_2^{\nu}\ \miu\ \left(   \max_{t_1 \in \left[\frac{t_{k_1}}{w}-x_1,\, \frac{t_{k_1+1}}{w}-x_1\right]} \!\! t^2_1\, +\, ... +\, \max_{t_n \in \left[\frac{t_{k_n}}{w}-x_n,\, \frac{t_{k_n+1}}{w}-x_n\right]} \!\! t^2_n \right)^{\nu/2}
$$
$$
\miu \left(  \max \left\{ \left( \frac{t_{k_1}}{w} -x_1 \right)^2, \left(\frac{t_{k_1+1}}{w}-x_1\right)^2\right\}\, + ...  +  \max \left\{ \left(\frac{t_{k_n}}{w}-x_n\right)^2, \left(\frac{t_{k_n+1}}{w}-x_n\right)^2 \right\}     \right)^{\nu/2}
$$
$$
\miu \frac{1}{w^{\nu}} \left(  \max \left\{ \left( wx_1- t_{k_1} \right)^2, \left(w x_{1}- t_{k_1+1}\right)^2\right\} + ... + \max \left\{ \left( wx_n- t_{k_n} \right)^2, \left(w x_{n}- t_{k_n+1}\right)^2\right\}   \right)^{\nu/2}
$$
$$
\miu \frac{1}{w^{\nu}} \left(  \max \left\{ \|w\xx- \tk \|_2^2, \|w \xx - \tk - \Delta_{\kk}\|_2^2\right\} + ... + \max \left\{ \| w\xx - \tk\|_2^2, \|w \xx - \tk - \Delta_{\kk}\|_2^2\right\}   \right)^{\nu/2},
$$
where $\Delta_{\kk}:=(\Delta_{k_1},\, ...,\, \Delta_{k_n})$. Now, recalling that each $\Delta_{k_i} \miu \Delta$, $i=1,...,n$, we can observe that
$$
\|w \xx - \tk - \Delta_{\kk}\|_2\ \miu \|w \xx - \tk\|_2\ +\ \sqrt{n}\ \Delta,
$$
then we finally obtain
$$
\max_{\ttt \in (R^w_{\kk}-\xx)}\|\ttt\|_2^{\nu}\ \miu\ w^{-\nu}n^{\nu/2} \left[\|w \xx - \tk\|_2\ +\ \sqrt{n}\ \Delta    \right]^{\nu}.
$$
Now, since $0 < \nu \miu \beta \miu 1$, we have that the function $x^{\nu}$ for $x \mau 0$ is concave, and then subadditive, so we can write
\begin{equation} \label{dis-important}
\max_{\ttt \in (R^w_{\kk}-\xx)}\|\ttt\|_2^{\nu}\ \miu\ w^{-\nu}\, n^{\nu/2}\ \left[\, \|w \xx - \tk\|^{\nu}_2\ +\  n^{\nu/2}\Delta^{\nu}\, \right].
\end{equation}
Then, by the inequalities in (\ref{stima_prop}) and (\ref{dis-important}) we obtain that
$$
I_1\ \miu\ w^{-\nu}\, n^{\nu/2}\, C \!\!\! \sum_{\|w\xx-\tk\|_2 \miu w\gamma/2}\left|\chi(w\xx-\tk)\right|\, \left[\, \|w \xx - \tk\|^{\nu}_2\ +\  n^{\nu/2}\Delta^{\nu}\, \right]
$$
$$
\miu\ w^{-\nu}\, n^{\nu/2}\, C  \left[  \sum_{\|w\xx-\tk\|_2 \miu w\gamma/2}\left|\chi(w\xx-\tk)\right|\|w \xx - \tk\|^{\nu}_2\ +\   n^{\nu/2}\Delta^{\nu}\, m_{0, \Pi^n}(\chi)  \right]
$$
$$
\hskip-4.5cm \miu\ w^{-\nu}\, n^{\nu/2}\, C  \left[  m_{\nu, \Pi^n}(\chi)\ +\   n^{\nu/2}\Delta^{\nu}\, m_{0, \Pi^n}(\chi)  \right].
$$
Now, by condition $(\chi 3)$ we have $m_{\beta, \Pi^n}(\chi) < +\infty$ and this implies that $m_{\nu, \Pi^n}(\chi) < +\infty$, for every $0 < \nu \miu \beta$; moreover by Lemma \ref{lemma1} (i) it turns out that $m_{0, \Pi^n}(\chi) < +\infty$, hence we can state 
$$
I_1\ =\ \mathcal{O}(w^{-\nu}),\ \hskip0.5cm as \hskip0.5cm w \to +\infty.
$$
Further, by Lemma \ref{lemma1} (ii),
$$
I_2\ \miu\ 2 \left\|f\right\|_{\infty} \sum_{\|w\xx-\tk\|_2 > w\gamma/2}\left|\chi(w\xx-\tk)\right|\ =\ \mathcal{O}(w^{-\beta}), \hskip0.5cm as \hskip0.5cm w\rightarrow +\infty,
$$
uniformly with respect to $\xx \in \R^n$, and finally from $(\chi 2)$
we obtain that $I_3 = \mathcal{O}(w^{-\mu})$, as $w \to +\infty$, uniformly with respect to $\xx \in \R^n$. Thus, we have shown that
$$
\left|(S_w f)(\xx)-f(\xx)\right|\ \miu\ I_1 + I_2 + I_3\ =\ \mathcal{O}(w^{-\nu}) + \mathcal{O}(w^{-\beta}) + \mathcal{O}(w^{-\mu}), \hskip0.5cm as \hskip0.5cm w\rightarrow +\infty,
$$
uniformly with respect to $\xx \in \R^n$ and therefore we finally obtain that
$$
\left\|S_w f-f\right\|_{\infty}\ =\ \mathcal{O}(w^{-\eps}), \hskip0.5cm as \hskip0.5cm w\rightarrow +\infty,
$$
where $\eps:=\min\left\{ \nu,\, \beta,\, \mu \right\}$.
\vskip0.2cm

  Let now $f \in Lip_{\infty}(\nu)$, with $\beta \miu \nu \miu 1$ be fixed. Since 
$Lip_{\infty}(\nu) \subseteq Lip_{\infty}(\beta)$, then for the previous case we can claim that 
$$
\left\|S_w f-f\right\|_{\infty}\ =\ \mathcal{O}(w^{-\eps}), \hskip0.5cm as \hskip0.5cm w\rightarrow +\infty,
$$
where $\eps:=\min\left\{ \beta,\, \mu \right\} = \min\left\{ \nu,\, \beta,\, \mu \right\}$, since $\nu \mau \beta$.
\vskip0.2cm
Finally, we consider $\chi$ satisfying condition $(\chi 3)$ for $\beta >1$. By the above considerations, $m_{\beta, \Pi}(\chi) < +\infty$ implies $m_{1,\Pi}(\chi)< +\infty$, which means that $\chi$ satisfies condition $(\chi 3)$ also for $\beta=1$, and so this case can be reduced to the previous step, and the assertion follows.
\end{proof}
%

%%%%%%%%%%%%%%%%%%%%%%%%%

\section{Order of approximation in Orlicz spaces $L^{\varphi}(\R^n)$} \label{orlicz}

  In order to study the behavior of the sampling Kantorovich operators when not necessary continuous signals (such as images) should be reconstructed, and to
study their degree of approximation, we consider the case of multivariate signal belonging to the general setting of Orlicz spaces $L^{\varphi}(\R^n)$, where $\varphi$ is a convex $\varphi$-function. We first recall the following modular continuity property for $S_w$.
\begin{theorem} \label{modular_cont}
Let $\chi$ be a kernel. For every $f \in L^{\varphi}(\R^n)$, there holds
$$
I^{\varphi}[\lambda S_w f]\ \miu\ \frac{\|\chi \|_1}{\delta^n m_{0,\Pi^n}(\chi)}I^{\varphi}[\lambda m_{0,\Pi^n}(\chi) f] \hskip1cm (\lambda>0),
$$
for every $w> 0$. In particular, $S_w f \in L^{\varphi}(\R^n)$ whenever $f \in L^{\varphi}(\R^n)$.
\end{theorem}
\noindent For a proof of Theorem \ref{modular_cont}, see \cite{COVI1}. The above theorem shows that the map $S_w: L^{\varphi}(\R^n) \to L^{\varphi}(\R^n)$ is well-defined and continuous with respect to the modular topology (\cite{MU1}).

  Now, we establish the following result which gives a rate of approximation for the sampling Kantorovich operators in Orlicz spaces.
\begin{theorem} \label{order_orlicz_spaces}
Let $\chi$ be a kernel and $f \in L^{\varphi}(\R^n) \cap Lip_{\varphi}(\nu)$, $0< \nu \miu 1$. Suppose in addition that there exist $\theta$, $\gamma >0$ such that
\begin{equation} \label{ip_chi_tau}
\int_{\|\ttt\|_2 \miu \gamma} w^n |\chi(w\ttt)|\ \|\ttt\|_2^{\nu}\ d\ttt\ =\ \mathcal{O}(w^{-\theta}), \hskip0.5cm as \hskip0.5cm w \to +\infty.
\end{equation}
Then there exists $\lambda >0$ such that
$$
I^{\varphi}[\lambda(S_wf - f)]\ =\ \mathcal{O}(w^{-\eps}), \hskip0.5cm as \hskip0.5cm w \to +\infty,
$$
with $\eps := \min\left\{\theta,\ \nu,\ \mu,\ \alpha\right\}$, where $\mu $, $\alpha>0$ are the constants of conditions $(\chi 2)$ and $(\chi4)$ respectively.
\end{theorem}
\begin{proof}
By the assumption $f \in L^{\varphi}(\R^n) \cap Lip_{\varphi}(\nu)$, $0<\nu \miu 1$, we have that $I^{\varphi}[\lambda_1 f] < +\infty$,
and
$$
I^{\varphi}[\lambda_2(f(\cdot)-f(\cdot +\ttt))]\ =\ \mathcal{O}(\|\ttt\|_2^{\nu}), \hskip0.5cm as \hskip0.5cm \|\ttt\|_2 \to 0,
$$
for some $\lambda_1$, $\lambda_2>0$. More in detail, there exist $M_1$, $\overline{\gamma}>0$ such that 
$$
I^{\varphi}[\lambda_2(f(\cdot)-f(\cdot +\ttt))]\ \miu\ M_1 \|\ttt\|_2^{\nu},
$$
for every $\|\ttt\|_2 \miu \overline{\gamma}$.
Now, by the properties of the convex modular functional $I^{\varphi}$, for $\lambda>0$ we can split the term $I^{\varphi}[\lambda(S_wf - f)]$ as follows:
\begin{eqnarray*}
&& I^{\varphi}[\lambda(S_wf - f)]\ =\ \int_{\R^n} \varphi(\lambda|(S_wf)(\xx)-f(\xx)|)\ d\xx \\
&\miu& \frac{1}{3} \left\{ \int_{\R^n} \varphi\left(3\lambda\left|(S_wf)(\xx)- \sum_{\kk\in \Z^n}\chi(w\xx-\tk)\, \frac{w^n}{A_{\kk}}\int_{R^w_{\kk}}f\left(\uu+\xx-\frac{\tk}{w}\right)\, d\uu\, \right|\right)\ d\xx     \right.\\
&+& \int_{\R^n} \varphi\left(3\lambda \left|\sum_{\kk \in \Z^n}\chi(w\xx-\tk)\, \frac{w^n}{A_{\kk}} \int_{\Rkw}f\left(\uu+\xx-\frac{\tk}{w}\right)\, d\uu\ -\ f(\xx) A_w(\xx)\, \right| \right)\ d\xx \\
&+& \left. \int_{\R^n} \varphi\left(3\lambda\left|f(\xx)A_w(\xx) - f(\xx)\right|\right)\ d\xx\ \right\}=:\ \frac{1}{3} \left\{ J_1\ +\ J_2\ +\ J_3 \right\}.
\end{eqnarray*}
 In order to estimate the above terms, we begin considering the first one, namely
$J_1$. Applying Jensen's inequality (see, e.g., \cite{COSP11}) and Fubini-Tonelli theorem, 
\begin{eqnarray*}
&J_1& = \int_{\R^n} \varphi\left(3\lambda\left|(S_wf)(\xx)- \sum_{\kk\in \Z^n}\chi(w\xx-\tk)\, \frac{w^n}{\Ak} \int_{\Rkw}f\left(\uu+\xx-\frac{\tk}{w}\right)d\uu\right|\right)\ d\xx \\
&\miu& \frac{1}{m_{0,\Pi^n}(\chi)} \int_{\R^n}\sum_{\kk\in \Z^n}|\chi(w\xx-\tk)|\varphi \left( 3\lambda m_{0,\Pi^n}(\chi)  \frac{w^n}{\Ak} \int_{\Rkw}| f(\uu)-f(\uu+\xx-\frac{t_k}{w})|d\uu\right) d\xx \\
&\miu& \frac{1}{m_{0,\Pi^n}(\chi)}\sum_{\kk\in \Z^n} \int_{\R^n}|\chi(w\xx-\tk)|\varphi \left( 3\lambda m_{0,\Pi^n}(\chi)  \frac{w^n}{\Ak} \int_{\Rkw}| f(\uu)-f(\uu+\xx-\frac{\tk}{w})|d\uu \right) d\xx.
\end{eqnarray*}
By the change of variable $\ttt=\xx-\tk/w$, applying Fubini-Tonelli theorem and Jensen's inequality again, we may obtain the following:
$$
J_1\ \miu\ \frac{1}{m_{0,\Pi^n}(\chi)} \int_{\R^n}|\chi(w\ttt)|\sum_{\kk\in \Z^n}\varphi \left( 3\lambda m_{0,\Pi^n}(\chi) \, \frac{w^n}{\Ak} \int_{\Rkw}| f(\uu)-f(\uu+\ttt)|d\uu\right) d\ttt 
$$
$$
\miu\ \frac{1}{m_{0,\Pi^n}(\chi)} \int_{\R^n}|\chi(w\ttt)| \left\{  \sum_{\kk\in \Z^n}\frac{w^n}{\Ak} \int_{\Rkw} \varphi \left( 3\lambda m_{0,\Pi^n}(\chi)\, | f(\uu)-f(\uu+\ttt)|\right)\, d\uu \right\} d\ttt $$
$$
\miu\ \frac{1}{m_{0,\Pi^n}(\chi)\ \delta^n}\int_{\R^n}w^n|\chi(w\ttt)|\left\{ \sum_{k\in \Z}\int_{\Rkw} \varphi ( 3\lambda m_{0,\Pi^n}(\chi) \, | f(\uu)-f(\uu+\ttt)|)\, d\uu \right\} d\ttt
$$
$$
\hskip-1cm =\ \frac{1}{m_{0,\Pi^n}(\chi)\ \delta^n}\int_{\R^n}w^n|\chi(w\ttt) | \left\{ \int_{\R^n}\, \varphi ( 3\lambda m_{0,\Pi^n}(\chi) \, | f(\uu)-f(\uu+\ttt)|)\, d\uu \right\} d\ttt 
$$
$$
=\ \frac{1}{m_{0,\Pi^n}(\chi)\ \delta^n} \left\{ \int_{\|\ttt\|_2|\miu \widetilde{\gamma}} w^n|\chi(w\ttt)| \left( \int_{\R^n}\varphi ( 3\lambda m_{0,\Pi^n}(\chi) \, | f(\uu)-f(\uu+\ttt)|)\, d\uu \right) d\ttt\ \ +  \right. 
$$
$$
\hskip-2cm +\ \left. \int_{\|t\|_2 >\widetilde{\gamma}} w^n |\chi(w\ttt)| \left( \int_{\R^n}\varphi ( 3\lambda m_{0,\Pi^n}(\chi) \, | f(\uu)-f(\uu+\ttt)|)\, d\uu \right) \, d\ttt \right\} 
$$
$$
\hskip-7cm =\ \frac{1}{m_{0,\Pi^n}(\chi)\ \delta^n} \, \left\{ J_{1,1}\ +\ J_{1,2} \right\},
$$
with $\widetilde{\gamma}:= \min \left\{ \gamma,\ \overline{\gamma}\right\}$, where $\gamma >0$ is the constant of condition (\ref{ip_chi_tau}). Now, without any loss of generality, we can choose $\lambda>0$ sufficiently small, such that:
$$
\lambda \miu\ \min\left\{\lambda_1/(3M_2),\ \lambda_1/(6m_{0,\Pi^n}(\chi)),\ \lambda_2 / (3m_{0,\Pi^n}(\chi)),\ \lambda_2 \delta^n /(3 \Delta^n m_{0,\Pi^n}(\chi))\right\},
$$
where $M_2>0$ is a suitable positive constant obtained from condition $(\chi 2)$, i.e.,
$$
|\sum_{k \in \Z} \chi(w\xx-\tk) - 1 |\ \miu M_2\, w^{-\mu},
$$
uniformly with respect to $\xx \in \R^n$ and for sufficiently large $w>0$. Now, recalling that $f \in Lip_{\varphi}(\nu)$, by condition (\ref{ip_chi_tau}) it is easy to deduce the following estimate:
\begin{eqnarray*}
J_{1,1} &\miu& \int_{\|\ttt\|_2 \miu \widetilde{\gamma}}w^n\, |\chi(w\ttt)|\, \left[\int_{\R^n}\varphi ( \lambda_2 | f(\uu)-f(\uu+\ttt)|)\, d\uu \right] d\ttt \\
&\miu&\ M_1 \int_{\|\ttt\|_2 \miu \widetilde{\gamma}}w^n\, |\chi(w\ttt)|\ \|\ttt\|_2^{\nu}\ d\ttt\ =\ \mathcal{O}(w^{-\theta}),\ \hskip0.5cm as \hskip0.5cm w \to +\infty,
\end{eqnarray*}
while, in case of $J_{1,2}$, by the convexity of $\varphi$, we have
\begin{eqnarray*}
&& J_{1,2}\ \miu \\
&\miu& \!\! \int_{\|\ttt\|_2>\widetilde{\gamma}} \!\!\!\!\!\! w^n\, |\chi(w\ttt)| \frac{1}{2} \left[ \int_{\R^n}\varphi ( 6\lambda m_{0,\Pi^n}(\chi)  | f(\uu)|)\, d\uu + \int_{\R^n}\varphi ( 6\lambda m_{0,\Pi^n}(\chi) | f(\uu+\ttt)|)\, d\uu \right] d\ttt.
\end{eqnarray*}
Observing that
$$
\int_{\R^n}\varphi ( 6\lambda m_{0,\Pi^n}(\chi)  | f(\uu+\ttt)|)\ d\uu\ =\ \int_{\R^n}\varphi ( 6\lambda m_{0,\Pi^n}(\chi)  | f(\uu)|)\ d\uu,
$$
for every $\ttt \in \R^n$, it turns out
$$
J_{1,2}\ \miu\ \int_{\|\ttt\|_2>\widetilde{\gamma}} w^n\, |\chi(w\ttt)|\int_{\R^n}\varphi (6\lambda m_{0,\Pi^n}(\chi) | f(\uu)|)\, d\uu 
$$
$$
\hskip-0.45cm =\ \! I^{\varphi}[6\lambda m_{0,\Pi^n}(\chi) f] \ \int_{\|\ttt\|_2>\widetilde{\gamma}}  w^n\, |\chi(w\ttt)|\, d\ttt.
$$
Therefore, by all the above inequality, assumptions and considerations, we deduce that
$$
J_{1,2}\ =\ \mathcal{O}(w^{-\alpha}),\hskip0.5cm as \hskip0.5cm w \to +\infty.
$$
Now, we estimate $J_2$. Setting $\ttt=\uu-\tk/w$ we have
$$
J_2\ \miu\ \int_{\R^n}\varphi\left(3\lambda \left| \sum_{\kk \in \Z^n} \chi(w\xx-\tk) \left[ \frac{w^n}{\Ak}\int_{\Rkw}f(\uu+\xx-\frac{\tk}{w})\ d\uu-f(\xx) \right]   \right|\right)\ d\xx 
$$
$$
=\ \int_{\R^n}\varphi\left(3\lambda \left| \sum_{\kk \in \Z^n} \chi(w\xx-\tk) \left[  \frac{w^n}{\Ak}\int_{(\Rkw - \tk/w)}f(\xx+\ttt)\ d\ttt\ -\ f(\xx) \right]   \right|\right)\ d\xx,
$$
where $(\Rkw - \tk/w) := [0,\, \Delta_{k_1}/w] \times ... \times [0,\, \Delta_{k_n}/w]$, for every $\kk \in \Z^n$ and $w>0$. Thus,
$$
J_2\ \miu\  \int_{\R^n}\varphi\left(3\lambda \left| \sum_{\kk \in \Z^n} \chi(w\xx-\tk) \frac{w^n}{\Ak}\int_{(\Rkw - \tk/w)}[f(\xx+\ttt)-f(\xx)]\, d\ttt   \right|\right)\ d\xx
$$
$$
\miu\ \int_{\R^n}\varphi\left(3\lambda \left[ \sum_{\kk \in \Z^n} \left| \chi(w\xx-\tk)\right| \frac{w^n}{\delta^n}\int_{(\Rkw - \tk/w)}|f(\xx+\ttt)-f(\xx)|\ d\ttt  \right] \right)\ d\xx
$$
$$
\miu\ \int_{\R^n}\varphi\left(3\lambda \left[ \sum_{\kk \in \Z^n} |\chi(w\xx-\tk)| \frac{w^n}{\delta^n} \int_{(\Delta_{w})}|f(\xx+\ttt)-f(\xx)|\ d\ttt   \right]\right)\ d\xx,
$$
where $(\Delta_{w}):=[0,\, \Delta/w] \times ... \times [0,\, \Delta/w]$. Then, using Jensen's inequality, Fubini-Tonelli theorem, and since $f \in Lip_{\varphi}(\nu)$, $0<\nu \miu 1$, for sufficiently large $w>0$ we can write
$$
J_2\ \miu\ \int_{\R^n}\varphi\left(3\lambda\, m_{0,\Pi^n}(\chi)\ \frac{w^n}{\delta^n} \int_{(\Delta_{w})}|f(\xx+\ttt)-f(\xx)|\ d\ttt \, \right)\ d\xx,
$$
$$
\miu\ \int_{\R^n}\frac{w^n}{\Delta^n} \left[ \int_{(\Delta_{w})}\varphi\left(3\lambda m_{0,\Pi^n}(\chi) \frac{\Delta^n}{\delta^n}\, |f(\xx+\ttt)-f(\xx)|\right)\ d\ttt \right] d\xx,
$$
$$
=\ \frac{w^n}{\Delta^n} \int_{(\Delta_{w})}  \left[ \int_{\R^n} \varphi\left(3\lambda m_{0,\Pi^n}(\chi) \frac{\Delta^n}{\delta^n}\, |f(\xx+\ttt)-f(\xx)|\right)\ d\xx\, \right]  d\ttt
$$
$$
\hskip-2.3cm \miu\ \frac{w^n}{\Delta^n} \int_{(\Delta_{w})}  \left[ \int_{\R^n} \varphi\left(\lambda_2\, |f(\xx+\ttt)-f(\xx)|\right)\ d\xx\, \right]  d\ttt
$$
$$
\hskip-6cm \miu\ M_1\, \frac{w^n}{\Delta^n} \int_{(\Delta_{w})} \|\ttt\|_2^{\nu}\,  d\ttt,
$$
being $\|\ttt\|_2 \miu \overline{\gamma}$. By the change of variable $\ttt=\uu/w$, and denoting by $(\overline{\Delta}):=[0,\, \Delta] \times ... \times [0, \Delta]$, we have that
$$
J_2\ \miu\ \frac{M_1}{\Delta^n} \int_{(\overline{\Delta})}\|\uu/w\|_2^{\nu}\ d\uu\ =\ w^{-\nu}\, \frac{M_1}{\Delta^n} \int_{(\overline{\Delta})}\|\uu\|_2^{\nu}\ d\uu\  =:\ Cw^{-\nu},
$$
for sufficiently large $w>0$, i.e., $J_2=\mathcal{O}(w^{-\nu})$, as $w \to +\infty$.

  Finally, we proceed by estimating the last term of the initial inequality, namely $J_3$. Recalling condition $(\chi 2)$ and the convexity of $\varphi$, for sufficiently large $w>0$ it is easy to obtain
\begin{eqnarray*}
J_3\ &\miu&\ \int_{\R^n} \varphi\left(3\lambda\left|f(\xx)||A_w(\xx) - 1|\right|\right)\ d\xx\ \miu\ \int_{\R^n} \varphi\left(3\lambda M_2 w^{-\mu}\left|f(\xx)|\right|\right)\ d\xx \\
&\miu&\ w^{-\mu} I^{\varphi}[3M_2\lambda f] \miu\ w^{-\mu} I^{\varphi}[\lambda_1 f]\ <\ +\infty. 
\end{eqnarray*}
Thus, combining all the above estimates, we can conclude that
$$
I^{\varphi}[\lambda(S_wf-f)]\ =\ \mathcal{O}(w^{-\eps}), \hskip0.5cm as \hskip0.5cm w \to\ +\infty,
$$
where $\eps:= \min\left\{\theta,\ \nu,\ \mu,\ \alpha\right\}$.
\end{proof}
  The assumptions required in the above theorem are rather standard when the problem of the rate of approximation for a family of linear discrete operators is studied in Orlicz spaces. In particular, condition (\ref{ip_chi_tau}) represents a relation between the kernel of $S_w$ and the Lipschitz class under consideration. Condition (\ref{ip_chi_tau}) is clearly satisfied when, for instance, the kernel $\chi$ is with compact support. Denoting by $B(\underline{0}, R)$ the closed ball centered in the origin and with radius $R>0$, if $supp\ \chi \subset B(\underline{0}, R)$, we have
\begin{equation} \label{nu+1}
\int_{\|\ttt\|_2\miu \gamma}w^n\ |\chi(w\, \ttt)|\ \|t\|_2^{\nu}\ d\ttt\ \miu\ \int_{\|\uu\|_2\miu R} |\chi(\uu)|\ \|\uu/w\|_2^{\nu}\ d\uu\ =:\ K w^{-\nu},
\end{equation}
for sufficiently large $w>0$, i.e., $\theta=\nu$. Moreover, in the above case, $\chi$ satisfies condition $(\chi 4)$ for every $\alpha >0$.
We obtain the following
\begin{corollary} \label{cor4}
Let $\chi$ be a kernel with compact support. Moreover, let $f \in L^{\varphi}(\R^n) \cap Lip_{\varphi}(\nu)$, $0 < \nu \miu 1$. Then, there exists $\lambda>0$ such that
$$
I^{\varphi}[\lambda(S_wf-f)]\ =\ \mathcal{O}(w^{-\eps}), \hskip0.5cm as \hskip0.5cm w \to +\infty,
$$
with $\eps:=\min\left\{\nu,\ \mu\right\}$, where $\mu >0$ is the constant of condition $(\chi 2)$. 
\end{corollary}
Condition (\ref{ip_chi_tau}) is satisfied also in case of kernels with unbounded support if we require an additional condition on $\chi$, i.e., that the multivariate absolute moment 
\begin{equation} \label{semp}
m_{\nu}(\chi):= \int_{\R^n}|\chi(\uu)|\ \|\uu\|_2^{\nu} d\uu\ <\ +\infty,
\end{equation}
for some $0< \nu \miu 1$. Indeed, if condition (\ref{semp}) holds, for every $\gamma>0$ we have
\begin{equation} \label{nu+1}
\int_{\|\ttt\|_2\miu \gamma}w^n\, |\chi(w\ttt)|\, \|\ttt\|_2^{\nu}\ d\ttt\ \miu\ \int_{\|\uu\|_2\miu \gamma w} |\chi(\uu)|\ \|\uu/w\|_2^{\nu}\ d\uu\ \miu\ m_{\nu}(\chi)\ w^{-\nu},
\end{equation}
for every $w>0$, which shows that (\ref{ip_chi_tau}) holds for $\theta = \nu$. Thus we obtain the following
\begin{corollary} \label{cor3}
Let $\chi$ be a kernel with $m_{\nu}(\chi)<+\infty$, for some $0<\nu \miu 1$. Moreover, let $f \in L^{\varphi}(\R^n) \cap Lip_{\varphi}(\nu)$ be fixed. Then there exists $\lambda>0$ such that
$$
I^{\varphi}[\lambda(S_wf-f)]\ =\ \mathcal{O}(w^{-\eps}), \hskip0.5cm as \hskip0.5cm w \to +\infty,
$$
with $\eps :=\min\left\{\nu,\ \mu,\ \alpha\right\}$, where $\mu>0$ and $\alpha>0$ are the constants of conditions $(\chi 2)$ and $(\chi 4)$, respectively. 
\end{corollary}
\begin{remark} \rm
Examples of convex $\varphi$-functions generating Orlicz spaces, where the theory of multivariate sampling Kantorovich operators holds, are:\\
$\varphi_p(u) :=u^p$, $1 \miu p < +\infty$, which generates the well-known $L^p(\R^n)$ spaces, $\varphi_{\alpha, \beta}:= u^{\alpha}\log^{\beta}(u+e)$, for $\alpha \mau 1$, $\beta>0$, which gives rise to interpolation spaces and  finally, $\varphi_{\gamma}(u)=e^{u^{\gamma}}-1$, for $\gamma>0$, $u \mau 0$, in order to obtain the exponential spaces. It is well-known that the modular functional corresponding to $\varphi_p(u)$ is $I^{\varphi_p}[f]:=\|f\|^p_p$. The modular functionals corresponding to $\varphi_{\alpha, \beta}$ and $\varphi_{\gamma}$ are
$$
I^{\varphi_{\alpha, \beta}}[f] := \int_{\R^n} |f(\xx)|^{\alpha} \log^{\beta}(e+|f(\xx)|)\ d\xx,\ \hskip0.5cm (f \in M(\R^n)),
$$
and
$$
I^{\varphi_{\gamma}}[f] := \int_{\R^n} (e^{|f(\xx)|^{\gamma}}-1)\ d\xx,\ \hskip0.5cm (f \in M(\R^n)),
$$
respectively. The $L^{\alpha} \log^{\beta} L$-spaces (interpolation or Zygmund spaces generated by $\varphi_{\alpha, \beta}$), are widely used in the theory of partial differential equations, while the exponential spaces (generated by $\varphi_{\gamma}$) are important for embedding theorems between Sobolev spaces.
\end{remark}

%%%%%%

\section{Applications to special kernels} \label{applications}

One important fact in our theory is the choice of the kernels, which influence the order of approximation that can be achieved by our operators (see e.g., \cite{COVI3,COVI4} in one-dimensional setting).

  To construct, in general, kernels satisfying all the assumptions $(\chi_i)$, $i=1,...,4$, is not very easy.

 For this reason, here we show a procedure useful to construct examples using product of univariate kernels, see, e.g., \cite{BUFIST,COVI1,COVI2,CLCOMIVI1}. For the sake of simplicity, we consider only the case of uniform sampling scheme, i.e., $t_{\underline{k}}=\underline{k}$.

 Denote by $\chi_1, ..., \chi_n$, the univariate functions $\chi_i : \R \to \R$, $\chi_i \in L^1(\R)$ and are bounded in a neighborhood of $0 \in \R$, satisfying
the following assumptions:
\begin{equation} \label{beta}
m_{\beta,\Pi^1}(\chi_i)\ :=\ \sup_{x \in \R}\sum_{k \in \Z}\left|\chi_i(x-k)\right|\cdot |x-k|^{\beta}<\ +\infty,
\end{equation}
$i=1,...,n$, for some $\beta>0$; moreover
\begin{equation} \label{sing}
\sum_{k \in \Z}\chi_i(x-k) = 1,
\end{equation}
for every $x \in \R$, $i=1,...,n$, and for every $\widetilde{M}>0$
\begin{equation} \label{kjkjkbbb}
\int_{|x|>\widetilde{M}}w \, |\chi_i(w x)|\, dx\ =\ \mathcal{O}(w^{-\alpha}), \hskip0.5cm as \hskip0.5cm w \to +\infty,
\end{equation}
for every $i=1,...,n$ and for some $\alpha>0$.
\begin{remark} \label{remark3} \rm
Note that condition $(\ref{sing})$ is equivalent to 
\begin{displaymath}
\widehat{\chi}(k) :=\ \left\{
\begin{array}{l}
0, \hskip0.5cm k \in \Z\setminus \left\{0\right\}, \\
1, \hskip0.5cm k=0,
\end{array}
\right.
\end{displaymath}
where $\widehat{\chi}(v):=\int_{\R}\chi(u)e^{-ivu}\ du$, $v \in \R$, denotes the Fourier transform of $\chi$; see \cite{BUNE,BABUSTVI,COVI1,COVI3}.
\end{remark}
Now, setting
\begin{equation}
\chi(\underline{x})\ :=\ \prod_{i=1}^n\chi_i(x_i),  \hskip1cm \underline{x}=(x_1,...,x_n) \in \R^n,
\end{equation} 
we can prove that $\chi$ is a multivariate kernel for the operators $S_w$ satisfying all the assumptions of our theory. Indeed,
we have that $\chi \in L^1(\R^n)$ since
\begin{displaymath}
\int_{\R^n}\left|\chi(\xx)\right|\ d\xx\ =\ \int_{\R^n}\prod^n_{i=1}\left|\chi_i(x_i)\right|\ dx_1...dx_n
=\ \prod^n_{i=1}\int_{\R}\left|\chi_i(x_i)\right|\ dx_i\ =\ \prod^n_{i=1} \| \chi_i \|_1\  <\ +\infty.
\end{displaymath}
Moreover, it is also obviously bounded in a neighborhood of the origin, then condition $(\chi 1)$ holds.
Further, by condition (\ref{sing})
\begin{equation}
A_w(\xx)\ =\ \sum_{\kk \in \Z^n}\chi(w \xx-\kk)\ =\ \prod^n_{i=1} \hskip0.2cm \sum_{k_i \in \Z} \chi_i(w x_i-k_i)\ =\ 1,
\end{equation}
then $A_w(\xx)-1=0$, for every $\xx \in \R^n$ and $w>0$, i.e., $\chi$ satisfies condition $(\chi 2)$ for every $\mu>0$.
Concerning condition $(\chi 3)$ we have:
$$
m_{\beta,\Pi^n}(\chi)\ =\ \sup_{\xx \in \R^n}\sum_{\kk \in \Z^n}\left|\chi(\xx-\kk)\right|\cdot \|\xx-\kk\|_2^{\beta} \miu\ K\, \sup_{\xx \in \R^n}\sum_{\kk \in \Z^n}\left|\chi(\xx-\kk)\right|\cdot \|\xx-\kk\|^{\beta},
$$
where $\| \xx \| := \max\left\{|x_i|,\, i=1,...,n\right\}$, $\xx \in \R^n$, and $K>0$ is a suitable constant (we recall that all the norms are equivalent in $\R^n$). Denoting by $\kk_{[j]} \in \Z^{n-1}$ the vectors $\kk_{[j]}:=(k_1, .., k_{j-1},k_{j+1},...,k_n)$ and by $\Pi^{n-1}_{[j]}$ the restriction of the sequence $\Pi^n$ to the vectors $\kk_{[j]}$, with $\kk_{[j]} \in \Z^{n-1}$, then we can write 
$$
\hskip-2.4cm m_{\beta,\Pi^n}(\chi)\ \miu\ K \sup_{\xx \in \R^n} \left[ \sum_{\kk \in \Z^n}\left|\chi(\xx-\kk)\right|\cdot \left( \sum_{j=1}^n |x_j-k_j|^{\beta} \right) \right]
$$
$$
\hskip-1.3cm \miu\ K \sup_{\xx \in \R^n} \sum_{j=1}^n\left[ \sum_{\kk \in \Z^n}\left|\chi(\xx-\kk)\right|\cdot |x_j-k_j|^{\beta}  \right]
$$
$$
\miu\ K \sup_{\xx \in \R^n} \sum_{j=1}^n\left[ \sum_{\kk \in \Z^n}\left(\prod_{i=1}^n|\chi_i(x_i-k_i)|\right)\cdot |x_j-k_j|^{\beta}  \right]
$$
$$
\miu\ K \sup_{\xx \in \R^n} \sum_{j=1}^n\left[ \sum_{\kk_{[j]} \in \Z^{n-1}}\left(\prod_{\stackrel{i=1}{i \neq j}}^n|\chi_i(x_i-k_i)|\right)\sum_{k_j \in \Z}|\chi_j(x_j-k_j)|\cdot |x_j-k_j|^{\beta}  \right]
$$
$$
\miu\ K \sup_{\xx \in \R^n} \sum_{j=1}^n\left[ \sum_{\kk_{[j]} \in \Z^{n-1}}\left(\prod_{\stackrel{i=1}{i \neq j}}^n|\chi_i(x_i-k_i)|\right) m_{\beta, \Pi^1}(\chi_j)  \right]
$$
$$
\miu\ K \sup_{\xx \in \R^n} \sum_{j=1}^n\left[ \prod_{\stackrel{i=1}{i \neq j}}^n\left(\sum_{\kk_{[j]} \in \Z^{n-1}} |\chi_i(x_i-k_i)|\right) m_{\beta, \Pi^1}(\chi_j)  \right]
$$
$$
\miu K \sup_{\xx \in \R^n} \sum_{j=1}^n\left[ \prod_{\stackrel{i=1}{i \neq j}}^n m_{0, \Pi_{[j]}^{n-1}}(\chi_i) \cdot m_{\beta, \Pi^1}(\chi_j)  \right] = K \sum_{j=1}^n\left[ \prod_{\stackrel{i=1}{i \neq j}}^n m_{0, \Pi_{[j]}^{n-1}}(\chi_i) \cdot m_{\beta, \Pi^1}(\chi_j)  \right]  < +\infty.
$$
Finally, for every $M>0$ there exists a suitable constant $\widetilde{M}>0$ such that
$$
\int_{\|\uu\|_2>M}w^n\, |\chi(w \uu)|\, d\uu\ \miu\ \int_{\|\uu\|>\widetilde{M}}w^n\, |\chi(w \uu)|\, d\uu\ =\ \int_{\|\uu\|>\widetilde{M}}\left[ \prod^n_{i=1} w\, |\chi_i(w u_i)|\right]\, d\uu
$$
$$
\miu\ \sum_{j=1}^n \left\{ \, \left[ \int_{\| \uu_{[j]}\| > \widetilde{M}}  \prod_{\stackrel{i=1}{i \neq j}}^n w |\chi_i(wu_i)|\, du_{[j]}\right] \cdot \int_{|u_j|>\widetilde{M}} w\, |\chi_j(w u_j)| \, du_j\right\}
$$
$$
\hskip-0.6cm \miu\ \sum_{j=1}^n \left\{ \, \left[ \int_{\| \ttt_{[j]}\| > w \widetilde{M}}  \prod_{\stackrel{i=1}{i \neq j}}^n |\chi_i(t_i)|\, d\ttt_{[j]}\right] \cdot \int_{|u_j|>\widetilde{M}} w\, |\chi_j(w u_j)| \, du_j\right\}
$$
$$
\hskip-1.3cm \miu\ \sum_{j=1}^n \left\{ \, \left[ \int_{\R^{n-1}}  \prod_{\stackrel{i=1}{i \neq j}}^n |\chi_i(t_i)|\, d\ttt_{[j]}\right] \cdot \int_{|u_j|>\widetilde{M}} w\, |\chi_j(w u_j)| \, du_j\right\}
$$
$$
\hskip-1.3cm \miu\ \sum_{j=1}^n \left\{ \, \left[\prod_{\stackrel{i=1}{i \neq j}}^n \int_{\R^{n-1}}   |\chi_i(t_i)|\, d\ttt_{[j]}\right] \cdot \int_{|u_j|>\widetilde{M}} w\, |\chi_j(w u_j)| \, du_j\right\}
$$
$$
\hskip-2.6cm \miu\ \sum_{j=1}^n \left\{ \, \left(\prod_{\stackrel{i=1}{i \neq j}}^n \|\chi_j\|_1\right) \cdot \int_{|u_j|>\widetilde{M}} w\, |\chi_j(w u_j)| \, du_j\right\},
$$
where we used the change of variable $w\, \uu_{[j]}=\ttt_{[j]}$;
then in correspondence to the constant $\alpha>0$ of condition (\ref{kjkjkbbb}) there holds: 
$$
\int_{\|\uu\|_2>M}w^n\, |\chi(w \uu)|\, d\uu =\ \mathcal{O}(w^{-\alpha}), \hskip0.5cm as \hskip0.5cm w\to +\infty,
$$
for every $M>0$, i.e., $\chi$ satisfies condition $(\chi 4)$ with $\alpha>0$.
Thus, we can say that $\chi$ is a multivariate kernel.

 Now, we will show some practical examples of multivariate kernels, constructed by product of univariate kernels.

  In what follows, we denote by
\begin{displaymath}
F(x)\ :=\ \frac{1}{2}\mbox{sinc}^2\left(\frac{x}{2}\right) \hskip0.5cm (x \in \R),
\end{displaymath}
the well-known one-dimensional Fej\'{e}r's kernel.
The $sinc(x)$ function is defined by
\begin{displaymath}
\mbox{sinc}(x)\ :=\ \left\{
\begin{array}{l}
\disp \frac{\sin \pi x}{\pi x}, \hskip1cm x \in \R\setminus \left\{0\right\}, \\
\hskip0.5cm 1, \hskip1.5cm x=0.
\end{array}
\right.
\end{displaymath}
It is easy to observe that the function $F$
is bounded, belongs to $L^1(\R)$ and satisfies the moment conditions $(\ref{beta})$ for every $0 < \beta \miu 1$. For the above properties see, e.g., \cite{BUNE,BABUSTVI,COVI3,COVI4}. Moreover, it is also possible to observe that
the Fourier transform of $F$ is given by (see \cite{BUNE})
\begin{displaymath}
\widehat{F}(v) :=\ \left\{
\begin{array}{l}
1-|v/\pi|,\ \hskip0.5cm |v|\miu \pi, \\
0,\ \hskip1.85cm |v|>\pi,
\end{array}
\right.
\end{displaymath}
and therefore condition $(\ref{sing})$ is fulfilled as a consequence of Remark \ref{remark3}. In addition, 
$$
\int_{|u|>\widetilde{M}}w\ F(w u)\ du\ \miu\ \frac{2}{\pi^2}w^{-1} \int_{|u|>\widetilde{M}}\frac{1}{u^2}\ du\ =:\ Kw^{-1},
$$
for every $\widetilde{M}>0$ and $w>0$, and hence condition (\ref{kjkjkbbb}) holds for $\alpha=1$. Finally, in \cite{BUNE} it is proved that the Fej\'{e}r's kernel satisfies the finiteness of the absolute moments, i.e., $m_{\nu}(F)<+\infty$, for every $0<\nu \miu 1$.
Then, according to the procedure described in this section, we can define by $\disp \mathcal{F}_n(\xx)= \prod^n_{i=1}F(x_i)$, $\underline{x}=(x_1,...,x_n) \in \R^n,$ the multivariate Fej\'{e}r's kernel, which satisfies the conditions upon a multivariate kernel with $m_{\nu}(\mathcal{F}_n)<+\infty$, for every $0<\nu \miu 1$. 
Then the multivariate sampling Kantorovich operators based on the Fej\'er's kernel, in case of the uniform sampling, take now the form
$$
(S^{\mathcal{F}_n}_w f)(\xx)\ =\ \sum_{\kk \in \Z^n}\left[w^n\int_{\Rkw}f(\uu)\ d\uu\right] \mathcal{F}_n\left(w\xx-\kk\right), \hskip0.5cm (\xx \in \R^n),
$$
for every $w>0$, where $f:\R^n \rightarrow \R$ is a locally integrable function such that the above series is convergent for every $\xx \in \R^n$. For $S^{\mathcal{F}_n}_w f$, from Theorem \ref{theorem1}, Theorem \ref{modular_cont} and Corollary \ref{cor3}, we can obtain respectively the following
\begin{corollary} \label{cor6.2}
Let $f \in Lip_{\infty}(\nu)$, with $0 < \nu \miu 1$. Then
$$
\left\|S^{\mathcal{F}_n}_w f-f\right\|_{\infty}\ =\ \mathcal{O}(w^{-\nu}), \hskip0.5cm as \hskip0.5cm w\rightarrow +\infty.
$$
In case of Orlicz spaces, for every $f \in L^{\varphi}(\R^n)$, there holds
$$
I^{\varphi}[\lambda S^{\mathcal{F}_n}_w f]\ \miu\ \frac{1}{\delta^n}\ I^{\varphi}[\lambda f], 
$$
for some $\lambda>0$ and for every $w> 0$, since $\| \mathcal{F}_n\|_1 = 1$ and $m_{0, \Pi^n}(\mathcal{F}_n) = 1$. In particular, $S^{\mathcal{F}_n}_w f \in L^{\varphi}(\R^n)$ whenever $f \in L^{\varphi}(\R^n)$.

\noindent Moreover, for any $f \in L^{\varphi}(\R^n)\cap Lip_{\varphi}(\nu)$, $0 < \nu \miu 1$, there exists $\lambda>0$ such that
$$
I^{\varphi}[\lambda(S^{\mathcal{F}_n}_wf-f)]\ =\ \mathcal{O}(w^{-\nu}), \hskip0.5cm as\ \hskip0.5cm w \to +\infty.
$$
\end{corollary}

  The Fej\'{e}r's kernel $\mathcal{F}_n$ provides an example of kernel with unbounded support. In this case, for a practical applications of the above reconstruction formula to any given signal with unbounded duration, one should evaluate the sampling series $S_w f$ at any fixed $\xx \in \R^n$, and this require to know an infinite number of mean values $w^n \int_{R^w_{\kk}}f(\uu)\ d\uu$. Clearly, in order to evaluate the operators $S_w f$ at $\xx$, the infinite sampling series must be truncated to a finite one, and this procedure leads to the so-called truncation error. However, if the signal $f$ has bounded duration, i.e., compact support, this problem does not arise. 

  In order to avoid the truncation error, kernels $\chi$ with compact support can be taken into consideration. 

  Noteworthy examples of such kernels can be constructed using the well-known, univariate central B-spline of order $k \in \N$, defined by
$$
M_k(x) :=\ \frac{1}{(k-1)!}\sum^k_{i=0}(-1)^i \left(\begin{array}{l} \!\! 
k\\
\hskip-0.1cm i
\end{array} \!\! \right)
\left(\frac{k}{2}+x-i\right)^{k-1}_+.
$$
We recall that, $(x)_+ := \max\left\{x,0\right\}$ denotes the positive part of $x \in \R$ (see \cite{BABUSTVI,VIZA1,COVI1}). 
We have that the Fourier transform of $M_k$ is given by
$$
\widehat{M_k}(v)\ :=\ \mbox{sinc}^n\left( \frac{v}{2 \pi} \right), \hskip0.5cm (v \in \R),
$$
and then, if we consider the case of the uniformly spaced sampling scheme, condition (\ref{sing}) is satisfied for every $\mu>0$ by Remark \ref{remark3}. Clearly, $M_k$ are bounded on $\R$, with compact support $[-k/2,k/2]$, and hence $M_k \in L^1(\R)$, for all $k \in \N$. Moreover, it is easy to deduce that conditions $(\ref{beta})$ and $(\ref{kjkjkbbb})$ are fulfilled for every $\beta>0$ and $\alpha >0$.
As in case of Fej\'{e}r's kernel, we
define the multivariate B-spline kernel of order $k \in \N^+$ by
\begin{displaymath}
\mathcal{M}^n_k(\xx)\ :=\ \prod^n_{i=1}M_k(x_i), \hskip1cm \underline{x}=(x_1,...,x_n) \in \R^n.
\end{displaymath}
Then the multivariate sampling Kantorovich operators based on the B-spline kernel of order $k$, in case of the uniform sampling scheme, take now the form
$$
(S^{\mathcal{M}^n_k}_w f)(x)\ =\ \sum_{\kk \in \Z^n}\left[w^n\int_{\Rkw}f(\uu)\ d\uu\right] \mathcal{M}^n_k\left(w\xx-\kk\right), \hskip0.5cm (\xx \in \R^n),
$$
for every $w>0$, where $f:\R^n \rightarrow \R$ is a locally integrable function such that the above series is convergent for every $\xx \in \R^n$. From Theorem \ref{theorem1}, Theorem \ref{modular_cont} and Corollary \ref{cor4}, we obtain the following
\begin{corollary} \label{cor6.2}
Let $f \in Lip_{\infty}(\nu)$, with $0 < \nu \miu 1$. Then
$$
\left\|S^{\mathcal{M}^n_k}_w f-f\right\|_{\infty}\ =\ \mathcal{O}(w^{-\nu}), \hskip0.5cm as \hskip0.5cm w\rightarrow +\infty.
$$
In case of Orlicz spaces, for every $f \in L^{\varphi}(\R^n)$, there holds
$$
I^{\varphi}[\lambda S^{\mathcal{M}^n_k}_w f]\ \miu\ \frac{1}{\delta^n}\ I^{\varphi}[\lambda f], 
$$
for some $\lambda>0$ and for every $w> 0$. In particular, $S^{\mathcal{M}^n_k}_w f \in L^{\varphi}(\R^n)$ whenever $f \in L^{\varphi}(\R^n)$.

\noindent Moreover, for any $f \in L^{\varphi}(\R^n)\cap Lip_{\varphi}(\nu)$, $0 < \nu \miu 1$, there exists $\lambda>0$ such that
$$
I^{\varphi}[\lambda(S^{\mathcal{M}^n_k}_wf-f)]\ =\ \mathcal{O}(w^{-\nu}), \hskip0.5cm as \hskip0.5cm w\to +\infty.
$$
\end{corollary}
\noindent For others useful examples of kernels see, e.g., \cite{BABUSTVI,BAMUVI,BUNE,VI1,CLCOMIVI1,COSP7,COSP8,COSP3,CO1,CO2,COSP10}.

\section*{{\normalsize Acknowledgment}}

{\normalsize The authors are members of the Gruppo  
Nazionale per l'Analisi Matematica, la Probabilit\'a e le loro  
Applicazioni (GNAMPA) of the Italian Istituto Nazionale di Alta Matematica (INdAM).
}

%%%%%%%%%%%%%%%%%%%%%%%%%


\begin{thebibliography}{99}


\bibitem{ANVI} Angeloni~L. and Vinti~G., Rate of approximation for nonlinear integral operators with applications to signal processing. {\it Differential Integral Equations} 18 (2005) (8), 855 -- 890.

\bibitem{BABUSTVI0} Bardaro~C., Butzer~P.L., Stens~R.L. and Vinti~G., Approximation of the Whittaker Sampling Series in terms of an Average Modulus of Smoothness covering Discontinuous Signals. {\it J. Math. Anal. Appl.} 316 (2006), 269 -- 306.

\bibitem{BABUSTVI} Bardaro~C., Butzer~P.L., Stens~R.L. and Vinti G., Kantorovich-Type Generalized Sampling Series in the Setting of Orlicz Spaces. {\it Sampl. Theory Signal Image Process.} 6 (2007) (1), 29 -- 52.

\bibitem{BABUSTVI2} Bardaro~C., Butzer~P.L., Stens~R.L. and Vinti G., Prediction by samples from the past with error estimates covering discontinuous signals. {\it IEEE Trans. Inform. Theory} 56 (2010) (1), 614 -- 633. 

\bibitem{BAMA} Bardaro~C. and Mantellini~I., Modular Approximation by Sequences of Nonlinear Integral Operators in Musielak-Orlicz Spaces. {\it Atti Sem. Mat. Fis. Univ. Modena, special issue dedicated to Professor Calogero Vinti} 46 (1998), 403 -- 425.

\bibitem{BAMA2} Bardaro~C. and Mantellini~I., On convergence properties for a class of Kantorovich discrete operators. {\it Numer. Funct. Anal. Optim.} 33 (2012) (4), 374 -- 396.

\bibitem{BAMUVI} Bardaro~C., Musielak~J. and Vinti~G., {\it Nonlinear Integral Operators and Applications.} New York, Berlin: De Gruyter Series in Nonlinear Analysis and Applications 9, 2003.

\bibitem{BAVI_1} Bardaro~C. and Vinti G., Some Inclusion Theorems for Orlicz and Musielak-Orlicz Type Spaces. {\it Annali di Matematica Pura e Applicata} 168 (1995), 189 -- 203.

\bibitem{BAVI2} Bardaro~C. and Vinti G., A general approach to the convergence theorems of generalized sampling series. {\it Applicable Analysis} 64 (1997), 203 -- 217.

\bibitem{BAVI7} Bardaro~C. and Vinti G., An Abstract Approach to Sampling Type Operators Inspired by the Work of P.L. Butzer - Part I - Linear Operators. {Sampl. Theory Signal Image Process.} 2 (2003) (3), 271 -- 296.

\bibitem{BEKA} Bezuglaya~L. and Katsnelson~V., The sampling theorem for functions with limited multi-band spectrum I. {\it Zeitschrift f\"{u}r Analysis und ihre Anwendungen} 12 (1993), 511 -- 534.

\bibitem{BUFIST} Butzer~P.L., Fisher~A. and Stens~R.L., Generalized sampling approximation of multivariate signals: theory and applications, {\it Note di Matematica}, 10 (1990) (1), 173-191.

\bibitem{BUNE}  Butzer~P.L. and Nessel~R.J., {\em Fourier Analysis and Approximation I}. New York-London: Academic Press 1971.

\bibitem{BURIST2} Butzer~P.L., Ries~S. and Stens~R.L.,  
Approximation of continuous and discontinuous functions by
generalized sampling series. {\it J. Approx. Theory} 50
(1987), 25 -- 39.

\bibitem{BUST1} Butzer~P.L. and Stens~R.L., Sampling theory for not necessarily band-limited functions: a historical overview. {\it SIAM Review} 34 (1992) (1), 40 -- 53.

\bibitem{BUST2} Butzer~P.L. and Stens~R.L., Linear prediction by samples from the past. In: {\it Advanced Topics in Shannon Sampling and Interpolation Theory} (editor R.J. Marks II), New York: Springer-Verlag 1993.

\bibitem{CLCOMIVI1} Cluni~F., Costarelli~D., Minotti~A.M. and Vinti~G., Multivariate sampling Kantorovich operators: approximation and applications to civil engineering. {\it EURASIP Open Library, Proceedings of SampTA 2013, 10th International  Conference on Sampling Theory and Applications, July 1st - July 5th, 2013, Jacobs University, Bremen} (2013), pp. 400 -- 403.

\bibitem{CLCOMIVI3} Cluni~F., Costarelli~D., Minotti~A.M. and Vinti~G.,  Applications of sampling Kantorovich operators to thermographic images for seismic engineering, {\em in print in: Journal of Computational Analysis and Applications} (2014).

\bibitem{CO1} Costarelli~D., Interpolation by neural network operators activated by ramp functions, {\em Journal of Mathematical Analysis and Application} 419 (1) (2014) 574-582.

\bibitem{CO2} Costarelli~D., Neural network operators: constructive interpolation of multivariate functions, {\em submitted} (2014).

\bibitem{COSP3} Costarelli~D. and Spigler~R., Approximation by series of sigmoidal functions with applications to neural networks, {\em in print in: Annali di Matematica Pura e Applicata} (2013) DOI: 10.1007/s10231-013-0378-y.

\bibitem{COSP7} Costarelli~D. and Spigler~R., Approximation results for neural network operators activated by sigmoidal functions, {\em Neural Networks} 44 (2013) 101 - 106. 

\bibitem{COSP8} Costarelli~D. and Spigler~R., Multivariate neural network operators with sigmoidal activation functions,
{\em Neural Networks} 48 (2013), 72 - 77.

\bibitem{COSP10} Costarelli~D. and Spigler~R., Convergence of a family of neural network operators of the Kantorovich type, {\em Journal of Approximation Theory} 185 (2014) 80 - 90.

\bibitem{COSP11} Costarelli~D. and Spigler~R., How sharp is Jensen's inequality ?, {\em submitted} (2014). 

\bibitem{COVI1} Costarelli~D. and Vinti~G., Approximation by multivariate generalized sampling Kantorovich operators in the setting of Orlicz spaces. {\it Bollettino U.M.I., Special volume dedicated to Prof. Giovanni Prodi} 4 (2011) (9),~445--~468.

\bibitem{COVI2} Costarelli~D. and Vinti~G., Approximation by nonlinear multivariate sampling Kantorovich type operators and applications to image processing. {\it Num. Funct. Anal. Opt.} 34 (2013) (8), 819 - 844.

\bibitem{COVI3} Costarelli~D. and Vinti~G., Order of approximation for sampling Kantorovich operators, in print in: {\it Journal of Integral Equations and Applications}, (2014).

\bibitem{COVI4} Costarelli~D. and Vinti~G., Order of approximation for nonlinear sampling Kantorovich operators in Orlicz spaces, {\it Commentationes Mathematicae, Special volume dedicated to Prof. Julian Musielak}, 5 (2) (2013), 171 - 192. 

\bibitem{DODSI}  Dodson~M.M. and Silva~A.M., Fourier Analysis and the Sampling Theorem. {\it  Proc. Ir. Acad.} 86 (1985) (A), 81 - 108.

\bibitem{DOVI} Donnini~C. and Vinti~G., Approximation by Means of Kantorovich Generalized Sampling Operators in Musielak-Orlicz spaces. {\it PanAmerican Mathematical J.} 18 (2008) (2), 1 - 18.

\bibitem{HIG2} Higgins~J.R., {\it Sampling Theory in Fourier and Signal Analysis: Foundations.} Oxford: Oxford Univ. Press 1996.

\bibitem{HIST} Higgins~J.R., and Stens~ R.L., {\it Sampling Theory in Fourier and Signal Analysis: advanced topics.} Oxford: Oxford Science Publications, Oxford Univ. Press, 1999.

\bibitem{JER} Jerry~A.J., The Shannon sampling-its various extensions and applications: a tutorial review. {\it Proc. IEEE} 65 (1977), 1565 -- 1596.

\bibitem{KOZ} Kozlowski~W.M., {\it Modular Function Spaces.} New York and Basel: Pure Appl. Math. Marcel Dekker, 1988.

\bibitem{KRA} Krasnosel'ski\v{i}~M.A. and Ruticki\v{i}~Ya.B., {\it Convex Functions and Orlicz Spaces.} Groningen: P. Noordhoff Ltd. The Netherlands, 1961.

\bibitem{MAL} Maligranda~L., {\it Orlicz Spaces and Interpolation.} Campinas: Seminarios de Matematica IMECC,  1989.

\bibitem{MU1} Musielak~J., {\it Orlicz Spaces and Modular Spaces.} Springer-Verlag, Lecture Notes in Math. 1034, 1983.

\bibitem{MUORL} Musielak~J. and Orlicz~W., {\it On modular spaces.} Studia Math. 28 (1959), 49 - 65.

\bibitem{RAO1} Rao~M.M. and Ren~Z.D., {\it Theory of Orlicz Spaces.}  New York-Basel-Hong Kong: Pure and Appl. Math. Marcel Dekker Inc. 1991.

\bibitem{RAO2} Rao~M.M. and Ren~Z.D., {\it Applications of Orlicz Spaces.} New York: Monographs and Textbooks in Pure and applied Mathematics 250, Marcel Dekker Inc. 2002.

\bibitem{RIST} Ries~S. and Stens~R.L., Approximation by generalized sampling series. In: {\it Constructive Theory of Functions'84}, Sofia, 1984, pp. 746 -- 756.

\bibitem{VEVI1} Ventriglia, F., and Vinti, G., A unified approach for the convergence of nonlinear Kantorovich type operators, {\em Communications on Applied  Nonlinear Analysis} 21  (2014) (2), 45 - 74.

\bibitem{CV1} Vinti~C., A Survey on Recent Results of the
Mathematical Seminar in Perugia, inspired by the Work of Professor P.L. Butzer, {\it Result. Math.} 34 (1998), 32 - 55.

\bibitem{VI0} Vinti, G., A general approximation result for nonlinear integral operators and applications to
signal processing. {\em Applicable Analysis} 79 (2001), 217 - 238.

\bibitem{VI1} Vinti~G., Approximation in Orlicz spaces for linear integral operators and applications. {\it Rendiconti del Circolo Matematico di Palermo Serie II} 76 (2005), 103 -- 127.

\bibitem{VIZA1} Vinti~G. and Zampogni~L., Approximation by means of nonlinear Kantorovich sampling type operators in Orlicz spaces. {\it J. Approx. Theory} 161 (2009), 511 - 528.

\bibitem{VIZA2} Vinti~G. and Zampogni~L., A Unifying Approach to Convergence of Linear Sampling Type Operators  in Orlicz Spaces. {\it Adv. Differential Equations} 16 (2011) (5-6), 573 - 600.

\bibitem{VIZA3} Vinti, G., and Zampogni, L., A unified approach for the convergence of linear Kantorovich-type operators, {\em in print in: Advanced Nonlinear Studies} (2014).

\end{thebibliography}
\end{document}